\documentclass[11pt,reqno,a4paper]{amsart}
\usepackage[english]{babel}
\usepackage[applemac]{inputenc}
\usepackage[T1]{fontenc}
\usepackage{palatino}
\usepackage{amsfonts}
\usepackage{amsmath}
\usepackage{amssymb}
\usepackage{amsthm}
\usepackage{graphicx}
\usepackage{wrapfig}
\usepackage{type1cm}
\usepackage[bf]{caption}
\usepackage{esint}
\usepackage{version}
\usepackage[active]{srcltx}  
\usepackage[usenames]{color}
\usepackage{caption}
\usepackage{tikz}
\usepackage{bbding, wasysym}
\usepackage{verbatim}
\usepackage{psfrag}

\newcommand{\R}{\mathbb{R}}

\newcommand{\N}{\mathbb{N}}

\newcommand{\eps}{\varepsilon}

\renewcommand{\epsilon}{\varepsilon}
\newcommand{\e}{\varepsilon}
\renewcommand{\rho}{\varrho}
\renewcommand{\phi}{\varphi}

\renewcommand{\i}{\mathtt{i}}
\renewcommand{\j}{\mathtt{j}}

\theoremstyle{plain}
\newtheorem{thm}{Theorem}
\newtheorem{theorem}[thm]{Theorem}
\newtheorem{lemma}{Lemma}

\newtheorem*{claim*}{Claim}

\theoremstyle{definition}
\newtheorem{definition}{Definition}

\newtheorem*{examples*}{Examples}
\newtheorem*{example*}{Example}

\newtheorem*{notations*}{Notations}
\newtheorem*{notation*}{Notation}

\numberwithin{equation}{section}
\numberwithin{thm}{section}
\numberwithin{lemma}{section}
\numberwithin{proposition}{section}
\numberwithin{cor}{section}
\numberwithin{claim}{section}
\numberwithin{definition}{section}
\numberwithin{conjecture}{section}
\numberwithin{example}{section}
\numberwithin{remark}{section}
\numberwithin{notations}{section}
\numberwithin{notation}{section}

\author{Pablo Shmerkin}
\address{University of Surrey, United Kingdom and Torcuato Di Tella University, Buenos Aires, Argentina}
\email{pshmerkin@utdt.edu}


\thanks{The author was partially supported by a Leverhulme Early Career Fellowship}

\title{Self-affine sets and the continuity of subadditive pressure}

\begin{document}

\begin{abstract}
The affinity dimension is a number associated to an iterated function system of affine maps, which is fundamental in the study of the fractal dimensions of self-affine sets. De-Jun Feng and the author recently solved a folklore open problem, by proving that the affinity dimension is a continuous function of the defining maps. The proof also yields the continuity of a topological pressure arising in the study of random matrix products. I survey the definition, motivation and main properties of the affinity dimension and the associated SVF topological pressure, and give a proof of their continuity in the special case of ambient dimension two.
\end{abstract}

\maketitle

\section{Introduction}

\label{sec:intro}

Let $\mathcal{F}=(f_1,\ldots, f_m)$ be a collection of contractive affine maps on some Euclidean space $\R^d$. That is, $f_i(x)=A_i x+t_i$, where $A_i\in\R^{d\times d}$ are linear maps, $t_i\in\R^d$ are translations, and $\|A_i\|<1$, where $\|\cdot\|$ denotes Euclidean operator norm (although any other operator norm would work equally well). It is well known that there exists a unique nonempty compact set $E=E(\mathcal{F})$ such that
\begin{equation} \label{eq:equality-self-affine}
E = \bigcup_{i=1}^m f_i(E) = \bigcup_{i=1}^m A_i E+t_i.
\end{equation}
Such sets are called \emph{self-affine}. The tuple $\mathcal{F}$ is termed as an \emph{iterated function system}, and $E$ is the \emph{attractor} or \emph{invariant set} of $\mathcal{F}$.

An important special case is that in which the maps $f_i$ are all similarities; in this case $E$ is known as a \emph{self-similar set}. It is known that for self-similar sets, Hausdorff, lower and upper box counting dimensions all agree. Moreover, if $s$ is the only real solution to $\sum_{i=1}^m r_i^s=1$, where $r_i$ is the similarity ratio of $f_i$, then $s$ is an upper bound for the Hausdorff dimension of $E$, and equals the Hausdorff dimension of $E$ under a number of ``controlled overlapping'' conditions, the strongest and simplest being the strong separation condition, which requires that the basic pieces $f_i(E)$ are mutually disjoint. The number $s$ is called the \emph{similarity dimension} of the system $\mathcal{F}$, and is clearly continuous, and indeed real-analytic, as a function of the maps $f_i$ (identified with the Euclidean space of the appropriate dimension).

The situation is dramatically more complex for general self-affine sets. It is well-known that the Hausdorff and box counting dimensions of self-affine sets may differ, and that each of them is a discontinuous function of the defining maps, even under the strong separation condition. Strikingly, it is not even known whether lower and upper box dimensions always coincide for self-affine sets. No general formula for either the Hausdorff or box counting dimension is known or expected to exist, again even in the strongly separated case. However, although the dimension theory of self-affine sets may appear at first sight like a bleak subject, many interesting and deep results have been obtained. Among these, K. Falconer's seminal paper \cite{Falconer92} has been highly influential. There, Falconer introduces a number $s=s(\mathcal{F})$ associated to an affine IFS $\mathcal{F}$, which we will term \emph{affinity dimension} (no standard terminology exists; the term singularity dimension is also often used). As a matter of fact, $s$ depends only on $A=(A_1,\ldots,A_m)$, i.e. the linear parts of the affine maps $f_i$, and is independent of the translations.

Falconer proved that the affinity dimension is always an upper bound for the upper box-counting, and therefore the Hausdorff, dimension of $E$, and in some sense, ``typically'' equals the Hausdorff dimension of $E$; his result is described in more detail in Section \ref{subsec:affinity-dim} below. The definition of the affinity dimension is rather more involved than the definition of similarity dimension, which it extends, and is postponed to Section \ref{subsec:affinity-dim}.

The question of whether the affinity dimension is continuous as a function of the generating maps has been a folklore open question in the fractal geometry community for well over a decade (I learned it from B. Solomyak around 2000), and was raised explicitly in \cite{FalconerSloan09}. Recently, together with De-Jun Feng \cite{FengShmerkin13} we proved that the answer is affirmative:

\begin{theorem} \label{thm:affinity-dim-continuous}
The affinity dimension $s$ is a continuous function of the linear maps $(A_1,\ldots, A_m)$.
\end{theorem}

A related but in some sense simpler result concerns the norms of matrix products. Again let $A=(A_1,\ldots, A_m)$ be a finite collection of invertible linear maps on $\R^d$. Given $s\ge 0$, define
\begin{equation} \label{eq:def-matrix-pressure}
M(A,s) = \lim_{n\to\infty}\frac1n \log\left( \sum_{(i_1\cdots i_k)} \| A_{i_k}\cdots A_{i_1}\|^s \right),
\end{equation}
where $\|\cdot\|$ is the standard Euclidean norm (the limit is easily seen to exist from subadditivity). The reader familiar with the thermodynamic formalism may note that this definition resembles the definition of topological pressure of a continuous potential on the full shift on $m$ symbols, except that here we consider norms of matrix products instead of Birkhoff sums; this is an instance of the topological pressure in the setting of the \emph{subadditive} thermodynamic formalism. This will be discussed in Section \ref{subsec:therm-form}.

The quantity $M(A,s)$ is rather natural. On one hand, in the thermodynamic setting it is closely linked to the Lyapunov exponent of an IID random matrix product (with respect to ergodic measures under the shift). On the other hand, the ``zero temperature limit'' as $s\to \infty$ is the joint spectral radius of the matrices $(A_1,\ldots,A_m)$, which is an important quantity in a wide variety of fields. Although the joint spectral radius is well-known to be continuous, it is far from clear from the definition whether $M(A,s)$ is always continuous. Together with D.J. Feng \cite{FengShmerkin13}, we have proved that it is:

\begin{theorem} \label{thm:matrix-pressure-continuity}
$M(A,s)$  is jointly continuous in $(A,s)$.
\end{theorem}
Although Theorems \ref{thm:affinity-dim-continuous} and \ref{thm:matrix-pressure-continuity} are in effect linear algebraic statements, the proofs make heavy use of dynamical systems theory, and in particular the variational principle for sub-additive potentials.

The goal of this survey is twofold. On one hand, it is an overview of the definition and main properties of the affinity dimension and the closely related singular value pressure, and the geometric reasons why it comes up naturally in the study of self-affine sets. On the other hand, it contains a full proof of Theorems \ref{thm:affinity-dim-continuous} and  \ref{thm:matrix-pressure-continuity} in the case of ambient dimension $d=2$ (for $d=1$, both results are trivial). The two-dimensional case captures many of the main ideas of the general case, while being technically much simpler.

I note that De-Jun Feng [private communication] has observed that a result of Bocker and Viana \cite{BockerViana10} on continuity of Lyapunov exponents for IID $\mathbb{R}^2$ matrix cocycles can be used to give a short alternative proof of Theorems \ref{thm:affinity-dim-continuous} and \ref{thm:matrix-pressure-continuity} in the case $d=2$. However, that proof does not generalize to any other dimensions.

\section{SVF, topological pressure, and affinity dimension}

\label{sec:SVF}

\subsection{Definition and basic properties of the SVF}

\label{subsec:SVF}

Recall that given a linear map $A\in \text{GL}_d(\R)$, its \emph{singular values} $\alpha_1(A)\ge \cdots \ge \alpha_d(A)> 0$ are the lengths of the semi-axes of the ellipsoid $A(B^d)$, where $B^d$ is the unit ball of $\R^d$. Alternatively, the singular values are the square roots of the eigenvalues of $A^*A$ (where $A^*$ is the adjoint of $A$). In particular, $\alpha_1(A)$ is nothing else than the Euclidean norm of $A$:
\[
\alpha_1(A) = \sup \{ \| A v\|: \|v\|=1\},
\]
where $\|v\|$ denotes the Euclidean norm of $v\in\R^d$. Likewise,
\[
\alpha_d(A) = \inf \{ \| Av\|: \|v\|=1\} = \|A^{-1}\|^{-1}.
\]
Also,
\[
\det(A) = \det(A^*A)^{1/2} = \prod_{i=1}^d \alpha_i(A).
\]
Given $s\in [0,d)$, we define the \emph{singular value function} (SVF) $\phi^s:\text{GL}_d(\R)\to (0,\infty)$ as follows. Let $m=\lfloor s\rfloor$. Then
\[
\phi^s(A) = \alpha_1(A)\cdots \alpha_m(A) \alpha_{m+1}(A)^{s-m}.
\]
An alternative way of expressing this is:
\begin{equation} \label{eq:SVF-using-exterior-norms}
\phi^s(A)= \|A\|_m^{m+1-s} \cdot \|A\|_{m+1}^{s-m},
\end{equation}
where
\[
\|A\|_k = \alpha_1(A)\cdots \alpha_k(A).
\]
The reason why \eqref{eq:SVF-using-exterior-norms} is useful is that $\|A\|_k$ is a sub-multiplicative seminorm ($\|A B\|_k\le \|A\|_k\|B\|_k$). Indeed, $\|A\|_k$ is the operator norm of $A$ when acting on the space of exterior $k$-forms. Alternatively, $\|A\|_k= \sup\{ \det(A|_\pi):\pi\in G(d,k)\}$ where $G(d,k)$ is the Grassmanian of $k$-dimensional linear subspaces of $\R^d$. As an immediate consequence of \eqref{eq:SVF-using-exterior-norms}, we get the following key property of the SVF:

\begin{lemma}[Sub-multiplicativity of the SVF] \label{lem:submultiplicatvity-of-SVF}
$\phi^s(AB)\le \phi^s(A)\phi^s(B)$ for all $A,B\in\text{GL}_d(\R)$ and $s\in[0,d)$.
\end{lemma}

It is also clear that $\phi^s(A)$ is jointly continuous in $s$ and $A$; it is also jointly real-analytic for non-integer $s$, but in general it is not even differentiable at $s=1,\ldots,d-1$ for a fixed $A$. We note also that $\phi^1(A)=\alpha_1(A)=\|A\|$ and $\lim_{s\to d}\phi^s(A)=\det(A)$. For completeness we define $\phi^s(A)=\det(A)^{s/d}$ for $s\ge d$, and note that this definition preserves all of the previous properties when $s\ge d$.

\subsection{SVF topological pressure}

\label{subsec:SVF-pressure}

Let $I=\{1,\ldots,m\}$. We denote by $I^*$ the family of finite words with symbols in $I$, and write $|\i|$ for the length of $\i\in I^*$. The space $X:=I^\N$ of right-infinite sequences is endowed with the left-shift operator $\sigma$, i.e. $\sigma(i_1 i_2\cdots) = (i_2 i_3 \cdots)$. Given $\i\in X$, the restriction of $\i$ to its first $k$ coordinates is denoted by $\i|k$. Finally, if $\j\in I^*$, then $[\j]\subset X$ is the family of all infinite sequences which start with $\j$.

Given $A=(A_1,\ldots, A_m)\in (\text{GL}_d(\R))^m$ and $\i=(i_1\cdots i_n)\in I^*$, we denote $A(\i) = A_{i_n}\cdots A_{i_1}$.  The next lemma introduces the main concept of this article.

\begin{lemma}
Given $A\in (\text{GL}_d(\R))^m$ and $s\ge 0$, let
\[
S_n(A,s) = \log\sum_{\i\in I^n} \phi^s(A(\i)).
\]
Then the limit
\begin{equation} \label{eq:def-topological-pressure}
P(A,s) := \lim_{n\to\infty} \frac{S_n(A,s)}{n}
\end{equation}
exists and equals $\inf_{n\ge 1} S_n(A,s)/n>-\infty$.
\end{lemma}
\begin{proof}
Lemma \ref{lem:submultiplicatvity-of-SVF} implies that the sequence $S_n=S_n(A,s)$ is subadditive, i.e. $S_{n+k}\le S_n+ S_k$. But it is well known that for any subadditive sequence $S_n$, the limit of $S_n/n$ exists and equals $\inf_{n\ge 1} S_n/n$. Finally, since $\phi^s(A)\ge \det(A)^s$ and we are assuming that the maps $A_i$ are invertible, one can easily check that
\[
\lim_{n\to\infty} \frac{S_n}{n} \ge \log\left(\sum_{i\in I} \det(A(\i))^s\right) > -\infty.
\]
\end{proof}

\begin{definition}
The function $P(A,s)$ defined in \eqref{eq:def-topological-pressure} is called the \emph{SVF topological pressure}.
\end{definition}

It is instructive to compare the definitions of $P(A,s)$ and $M(A,s)$ given in \eqref{eq:def-matrix-pressure}. Both quantities coincide for $0\le s\le 1$; for $s>1$, the definition of $P(A,s)$ takes into account different singular values of the matrix products involved. Since $\phi^s(A)\le \|A\|^s$, one always has $P(A,s)\le M(A,s)$.

The following lemma summarizes some elementary but important continuity properties of the topological pressure.

\begin{lemma} \label{lem:properties-SVF-pressure}
The following hold:
\begin{enumerate}
\item Given $A=(A_1,\ldots,A_m)\in(\text{GL}_d(\R))^m$, let
\[
\alpha_*=\min_{i\in I}\{ \alpha_d(A_i)\},\quad \alpha^*=\max_{i\in I}\{ \alpha_1(A_i)\}.
\]
Then
\[
(\log\alpha_*) \eps \le P(A,s+\eps)-P(A,s) \le (\log\alpha^*) \eps.
\]
\item For fixed $A=(A_1,\ldots,A_m)\in(\text{GL}_d(\R))^m$, the function $s\to P(A,s)$ is Lipschitz continuous; the Lipschitz constant is uniform in a neighborhood of $A$.
\item $P(A,s)$ is upper semicontinuous (as a function of both $A$ and $s$).
\end{enumerate}
\end{lemma}
\begin{proof}
Note that $\phi^s(B)\alpha_d(B)^\eps \le \phi^{s+\eps}(B) \le \phi^s(B)\alpha_1(B)^\eps$ for any $s,\eps>0$ and $B\in \text{GL}_d(\R)$. Also, $\alpha_1(B)=\|B\|$ is sub-multiplicative and $\alpha_d(B)=\|B^{-1}\|^{-1}$ is super-multiplicative. Combining these facts yields
\[
{n\eps}\log\alpha_*+ S_n(A,s) \le S_n(A,s+\e) \le n\eps\log\alpha^*+ S_n(A,s),
\]
which yields the first claim. The second claim is immediate from the first, and the fact that $\alpha^*,\alpha_*$ are continuous functions of $A$.

Finally, upper semicontinuity follows since $P(A,s)=\inf_{n\ge 1} S_n(A,s)/n$ is an infimum of continuous functions.
\end{proof}

In light of the previous lemma, it seems natural to ask whether $P(A,s)$ is not just upper semicontinuous but in fact continuous. As we will see in the next section, this question is closely linked to Theorem \ref{thm:affinity-dim-continuous}. Falconer and Sloan \cite{FalconerSloan09} proved continuity of $P$ at tuples of linear maps satisfying certain assumptions, and raised the general continuity problem. De-Jun Feng and the author \cite{FengShmerkin13} recently proved that continuity always holds:

\begin{theorem} \label{thm:continuity-pressure}
The map $(A,s)\to P(A,s)$ is continuous on $(\text{GL}_d(\R))^m\times [0,+\infty)$.
\end{theorem}
A proof of this theorem in dimension $d=2$ will be presented in Section \ref{sec:proof}.

\subsection{Affinity dimension and self-affine sets}

\label{subsec:affinity-dim}

So far, no assumptions have been made on the maps $A_i$, other than invertibility. However, the motivation for the study of the SVF topological pressure came from the theory of self-affine sets, and in this context the maps $A_i$ are strict contractions.

\begin{lemma}
If $A=(A_1,\ldots, A_m)\in(\text{GL}_d(\R))^m$ and $\|A_i\|<1$ for all $i\in I$, then $s\to P(A,s)$ is a continuous, strictly decreasing function of $s$ on $[0,\infty)$, and has a unique zero on $(0,\infty)$.
\end{lemma}
\begin{proof}
That $P(A,s)$ is continuous and strictly decreasing in $s$ follows immediately from Lemma \ref{lem:properties-SVF-pressure}, since $\alpha^*<1$ when the maps are strict contractions. By definition $P(A,0)=\log m>0$. On the other hand,
\[
P(A,s)\le \log\left(\sum_{i\in I} \phi^s(A_i)\right) \to -\infty\quad\text{ as }s\to \infty.
\]
Hence $P(A,\cdot)$ has a unique zero.
\end{proof}

\begin{definition}
Given $A=(A_1,\ldots, A_m)\in(\text{GL}_d(\R))^m$ with $\|A_i\|<1$ for all $i\in I$, its \emph{affinity dimension} is the unique positive root $s$ of the pressure equation $P(A,s)=0$.
\end{definition}

Note that Theorem \ref{thm:affinity-dim-continuous} is in fact an immediate corollary of Theorem \ref{thm:continuity-pressure}.

In the rest of this section we indicate why this number is relevant in the study of self-affine sets; this material is by now standard. To begin, let us recall the definition of Hausdorff dimension in terms of Hausdorff content:
$\dim_H(E)=\inf\{ s: \mathcal{H}_\infty^s(E)=0\}$, where
\[
\mathcal{H}_\infty^s(E) = \inf\left\{ \sum_i r_i^s: E\subset B(x_i,r_i) \right\}.
\]

Now suppose $E=\bigcup_{i\in I} A_i E+t_i$ is  the invariant set of the IFS $\{ A_i x+t_i\}_{i\in I}$. Since the maps $f_i(x)=A_i x+t_i$ are strict contractions, for all large enough $R$ the closed ball $B_R$ of radius $R$ and center at the origin is mapped into its interior by all the maps $f_i$. Let $E_0=B_R$ and define inductively $E_{k+1}=\bigcup_{i\in I} f_i(E_k)$. By our choice of $R$, it is easy to see inductively that $E_k$ is a decreasing sequence of compact sets; moreover, if we call $F=\bigcap_{k=0}^\infty E_k$, then one can check that $F=\bigcup_{i\in I} f_i(F)$. Thus $E=F$ by uniqueness.

The above discussion shows that $E$ is covered by $E_k$ for any $k$; moreover, by construction
\[
E_k = \bigcup_{i\in I^k} f_{i_1}\cdots f_{i_k}(B_R) =: \bigcup_{i\in I^k} U_{i_1\cdots i_k},
\]
where $U_{i_1\cdots i_k}$ is an ellipsoid with semi-axes $R \alpha_1(A_{i_1}\cdots A_{i_k})\ge \ldots\ge R \alpha_d(A_{i_1}\cdots A_{i_k})$. This shows that there is a natural cover of the self-affine sets by \emph{ellipsoids}. In order to estimate Hausdorff content (and hence Hausdorff dimension) effectively, one needs to find efficient coverings by \emph{balls}. What Falconer observed is that we can cover each ellipsoid efficiently by balls, in a way that depends on the dimension of the Hausdorff content we are trying to estimate. Namely, for each integer $1\le m<d$, we can cover an ellipsoid in $\R^d$ with semi-axes $\alpha_1\ge\cdots \ge\alpha_d$ by a parallelepiped with sides $2\alpha_1\ge \cdots\ge 2\alpha_d$. In turn, this can be covered by at most
\[
(4R\alpha_1/\alpha_m)\cdots (4R\alpha_{m-1}/\alpha_m)(4R)^{d-m+1}
\]
cubes of side length $\alpha_m$, each of which is contained in a ball of radius $\sqrt{d}\alpha_m$. It turns out that if we want to estimate $\mathcal{H}_\infty^s(E)$ by covering each of the ellipsoids that make up $E_k$ in this way, independently of each other, the optimal choice of $m$ is $\lfloor s\rfloor$. This particular choice yields (after some straightforward calculations) a bound
\[
\mathcal{H}_\infty^s(E) \le C_{R,d} \sum_{(i_1,\ldots,i_k)\in I^k} \phi^s(A_{i_1}\cdots A_{i_k}).
\]
From here one can deduce that if $P(A,t)<0$, then $\mathcal{H}_\infty^t(E)=0$, whence $\dim_H(E)\le t$. Letting $t\to s$, the affinity dimension, finally shows that $\dim_H(E)\le s$.

The argument above can be modified to reveal that the affinity dimension is also an upper bound for the upper box counting dimension of $E$. Thus, we can say that the affinity dimension is a \emph{candidate} to the (Hausdorff, or box-counting) dimension of a self-affine set, obtained by using the most natural coverings, and is always an upper bound for both the box-counting and Hausdorff dimension. In general, these natural coverings may be far from optimal. For example, many of the ellipsoids making up $E_k$ may overlap substantially or be aligned in such a way that it is far more efficient to cover them together rather than separately. Also, most of the cubes that we employed to cover each ellipsoid might not intersect $E$ at all. And indeed, it may happen that the Hausdorff dimension, and/or the box-counting dimension are strictly smaller than the affinity dimension; this is the case for many kinds of self-affine carpets, see e.g. \cite{Baranski07} and references therein. However, it is perhaps surprising that, as discovered by Falconer \cite{Falconer92}, \emph{typically} the Hausdorff and box-counting dimensions of self-affine sets do coincide with the affinity dimension, in the following precise way:

\begin{theorem} \label{thm:Falconer}
Let $A=(A_1,\ldots,A_m)$, with the $A_i$ invertible linear maps on $\R^d$. Assume further that $\|A_i\|<1/2$ for all $i\in I$. Given $t_1,\ldots,t_m\in\R^d$, denote by $E(t_1,\ldots,t_m)$ the self-affine set corresponding to the IFS $\{ A_i x+t_i\}_{i\in I}$.

Then for Lebesgue-almost all $(t_1,\ldots,t_m)\in\R^{md}$, the Hasudorff dimension of $E(t_1,\ldots,t_m)$ equals the affinity dimension of $A$.
\end{theorem}

We remark that Falconer proved the second part under the assumption $\|A_i\|<1/3$. Solomyak \cite{Solomyak98} later pointed out a modification in the proof that allows to replace $1/3$ by $1/2$. By an observation of Edgar \cite{Edgar92}, the bound $1/2$ is optimal. Since Falconer's pioneering work, many advances have been obtained in this direction. A natural question is whether one can give \emph{explicit} conditions under which the Hausdorff and/or box-counting dimensions equal the affinity dimension; this was achieved in \cite{Falconer92, HueterLalley95, KaenmakiShmerkin09}. In a different direction, Falconer and Miao \cite{FalconerMiao08} provided a bound on the dimension of exceptional parameters $(t_1,\ldots,t_m)$. For other recent directions in the study of the dimension of self-affine sets, see the survey \cite{Falconer13}.

\section{Further background}

\label{sec:background}

\subsection{Subadditive thermodynamic formalism}

\label{subsec:therm-form}

The topological pressure $P(\varphi)$ of a H\"{o}lder continuous potential $\varphi$ is a key component of the thermodynamic formalism, which in turn, as discovered by Bowen, is a formidable tool in the dimension theory of conformal dynamical systems. In the classical setting, the functional $P$ is continuous as a function of $\varphi$ in the appropriate topology.

It is well-known that the dimension theory of \emph{non-conformal} dynamical systems is far more difficult, and the classical thermodynamic formalism is no longer the appropriate tool. Instead, starting with the insights of Barreira \cite{Barreira96} and Falconer \cite{Falconer88}, a \emph{sub-additive} thermodynamic formalism has been developed. Both the thermodynamic formalism itself and its application to the dimension of invariant sets and measures is far more difficult in the non-conformal case. The proofs of Theorems \ref{thm:continuity-pressure} and \ref{thm:matrix-pressure-continuity} depend crucially on this subadditive thermodynamic formalism, and hence we review the main elements in this section.

We limit ourselves to potentials defined on the full shift on $m$ symbols $X=I^\N$. Let $\Phi=\{ \phi_n\}_{n=1}^\infty$ be a collection of continuous real-valued maps on $X$. We say that $\Phi$ is \emph{subadditive} if
\begin{equation} \label{eq:subadditive}
\varphi_{k+n}(\i) \le \varphi_k(\i)+\varphi_n(\sigma^k \i) \quad\text{for all  }\i\in X.
\end{equation}

Important examples of subadditive potentials, which will be relevant in the proofs of Theorems \ref{thm:continuity-pressure} and \ref{thm:matrix-pressure-continuity}, are
\begin{align}
\varphi_n(\i) &= s\log\|A_{i_n}\cdots A_{i_1}\|, \label{eq:potential-norm}\\
\varphi_n(\i) &= \phi^s(A_{i_n}\cdots A_{i_1}). \label{eq:potential-SVF}
\end{align}
We note that the order of the products is the reverse of the order usually considered in the IFS literature; the reason for this will become apparent later when we apply Oseledets' Theorem. Let $\mathcal{E}$ denote the set of probability measures ergodic and invariant under the shift $\sigma$. The thermodynamic formalism consists of three main pieces: the \emph{entropy} $h_\mu$ of a measure $\mu\in\mathcal{E}$, the \emph{topological pressure} $P(\Phi)$ of a subadditive potential $\Phi$, and the \emph{energy}  or \emph{Lyapunov exponent} $E_\mu(\Phi)$ of $\Phi$ with respect to a measure $\mu\in\mathcal{E}$. These are defined as follows:
\begin{align*}
h_\mu &= \lim_{n\to\infty} \frac{1}{n} \sum_{\i\in I^n} - \mu[\i] \log \mu[\i],\\
P(\Phi) &= \lim_{n\to\infty} \frac{1}{n}\log \sum_{\i\in I^n} \sup_{\j\in\i} \phi_n(\j),\\
E_\mu(\Phi) &= \lim_{n\to\infty} \frac{1}{n}\log \int \phi_n\,d\mu.
\end{align*}
By standard subadditivity arguments, all the limits exist. Moreover, if $\mu\in\mathcal{E}$, then for $\mu$-almost all $\i$,
\begin{align}
h_\mu &= \lim_{n\to\infty} \frac{-\log \mu[\i|n]}{n}, \label{eq:ShannonMcMillanBreiman}\\
E_\mu(\Phi) &= \lim_{n\to\infty} \frac{\log \phi_n(\i)}{n}. \label{eq:subadditive-ergodic-theorem}
\end{align}
The first equality is a particular case of the Sannon-McMillan-Breiman, while the second is a consequence of Kingman's subadditive ergodic theorem.

These quantities are related via the following \emph{variational principle} due to Cao, Feng and Huang \cite{CFH08}:
\begin{theorem}[\cite{CFH08}, Theorem 1.1] \label{thm:variational-principle}
If $\Phi$ is a subadditive potential on $X$, then
\[
P(\Phi) = \sup\left\{ h_\mu + E_\mu(\Phi) : \mu \in\mathcal{E} \right\}.
\]
\end{theorem}
Particular cases of the above, under stronger assumptions on the potentials, were previously obtained by many authors, see for example \cite{Mummert06, Barreira10} and references therein.

It is natural to ask whether the supremum in Theorem \ref{thm:variational-principle} is attained; measures which attain the supremum are known as \emph{equilibrium measures} or \emph{equilibrium states} (for the potential $\Phi$). While the existence of equilibrium measures in general is an open problem, K\"{a}enm\"{a}ki \cite{Kaenmaki04} proved that at least one equilibrium measure exists under fairly weak conditions, and in particular for the potentials given in \eqref{eq:potential-norm} and \eqref{eq:potential-SVF}. Unlike the classical setting, equilibrium measures do not need to be unique in the subadditive setting, not even in the locally constant case. Feng and K\"{a}enm\"{a}ki \cite{FengKaenmaki11} characterize all equilibrium measures for potentials of the form $\phi_n(\i) = s\log\|A_{i_1}\cdots A_{i_n}\|$.

\subsection{Oseledets' Multiplicative Ergodic Theorem}

\label{subsec:oseledets}

We recall a version of the Multiplicative Ergodic Theorem of Oseledets. For simplicity we state it only in dimension $d=2$; see e.g. \cite[Theorem 5.7]{Krengel85} for the full version.

\begin{theorem} \label{thm:oseledets}
Let $B_1,\ldots,B_m\in \text{GL}_2(\R)$, and for $\i\in X$ write
\[
B(\i,n)=B_{\i_n} B_{\i_{n-1}}\cdots B_{\i_1}.
\]
Further, let $\mu$ be a $\sigma$-invariant and ergodic measure on $X$. Then, one of the two following situations occur:

\begin{enumerate}
\item[(A)] (Equal Lyapunov exponents). There exists $\lambda\in\R$ such that for $\mu$-almost all $i$,
\[
\lim_{n\to\infty} \frac{\log |B(\i,n)v|}{n}=\lambda \quad\text{uniformly for } |v|=1.
\]
\item[(B)] (Distinct Lyapunov exponents). There exist $\lambda_1 > \lambda_2$ and measurable families $\{ E_1(\i)\}, \{ E_2(\i)\}$ of one-dimensional subspaces such that for $\mu$-almost all $\i$:
    \begin{enumerate}
    \item $\R^2 = E_1(\i)\oplus E_2(\i)$.
    \item $B_{i_1}E_j(\i) = E_j(\sigma\i)$, $j=1,2$.
    \item For all $v\in E_j(\i)\setminus\{0\}$, $j=1,2$,
\[
\lim_{n\to\infty} \frac{\log |B(\i,n)v|}{n}=\lambda_j.
\]
    \end{enumerate}
\end{enumerate}
\end{theorem}

\subsection{The cone condition}

\label{subsec:cone-condition}

The proof of Theorems \ref{thm:matrix-pressure-continuity} and Theorem \ref{thm:continuity-pressure} rely on finding a subsystem (after iteration) which is better behaved than the original one and captures almost all of its topological pressure. In the case of distinct Lyapunov exponents (with respect to a measure chosen from an application of the variatonal principle), the good behavior of this subsystem will consist in satisfying the (strict) cone condition: all the maps will send some fixed cone into its interior (except for the origin). Recall that a \emph{cone} $K\subset\R^d$ is a closed set such that $K\cap -K=\{0\}$ and $tx\in K$ whenever $t>0, x\in K$ (here $-K=\{-x:x\in K\}$).

This kind of cone condition is ubiquitous in the study of dynamical systems and associated matrix cocycles. In our situation, its usefulness will be derived from the following lemma.

\begin{lemma} \label{lem:cone-condition}
Let $K',K\subset\R^d$ be cones such that $K'\setminus\{0\}\subset \text{interior}(K)$. There exists a constant $c>0$ (depending on the cones) such that
\begin{equation} \label{eq:norm-cone-condition}
\|A\| \ge c\, \frac{\|A w\|}{\|w\|}
\end{equation}
for all $w\in K$ and all $A\in\R^{d\times d}$ such that $AK\subset K'$.

In particular, there is $c'>0$ such that if $A_1,A_2\in\R^{d\times d}$ are such that $A_j K\subset (K'\cup -K')$, $j=1,2$, then
\[
\| A_1 A_2\| \ge c' \|A_1\| \|A_2\|.
\]
\end{lemma}
\begin{proof}
Suppose \eqref{eq:norm-cone-condition} does not hold. Then, for all $n$ we can find a linear map $A_n$ of norm $1$ with $A_n(K)\subset K'$, and $w_n\in K'$ also of norm $1$, such that $\| A_n w_n\|< 1/n$. By compactness, this implies that there are a linear map $A$ on $V$ of norm $1$ (in particular nonzero) such that $A(K)\subset K'$, and a vector $w\in K'$ such that $Aw=0$. Now pick $u\in K$ such that $Au\neq 0$ and $w-u\in K$; this is possible since $K'\setminus\{0\}\subset\text{interior}(K)$. It follows that $A(w-u)=-Au\in -K'$, whence $Au\in K'\cap -K'$, contradicting that $K'$ is a cone.

For the second claim, we may assume (replacing $A_j$ by $-A_j$ if needed) that $A_j K\subset K'$ for $j=1,2$. The claim now follows from the first one, since for fixed $w\in K'$ of norm $1$,
\[
\| A_1 A_2 \| \ge c\| A_1 (A_2 w)\| \ge c^2 \| A_1\| \|A_2 w\| \ge c^3 \| A_1\| \|A_2\|.
\]
\end{proof}

A tuple $A=(A_1,\ldots, A_m)$ is said to satisfy the \emph{cone condition} if there exist cones $K', K$ such that $K'\setminus\{0\}$ is contained in the interior of $K$, and $A_i K\subset (K'\cup -K')$ for all $i\in I$. The relevance of this condition can be seen from the following lemma.

\begin{lemma} \label{lem:continuity-under-cone-condition}
If $A$ satisfies the cone condition, then $M$ is continuous on $\mathcal{U}\times [0,+\infty)$ for some neighborhood $\mathcal{U}$ of $A$, and the same holds for $P$ if  $d=2$.
\end{lemma}
\begin{proof}
We know from Lemma \ref{lem:properties-SVF-pressure} that $P$ is upper semicontinuous and Lipschitz continuous in $s$, with the Lipschitz constant locally uniformly bounded. The same arguments show that the same is true for $M$. Hence the task is to prove the claim with ``lower continuous'' in place of ``continuous'', with the value of $s$ fixed.

A trivial but key observation is that the cone condition is robust, in the following sense: if $A=(A_1,\ldots,A_m)$ satisfies the cone condition with cones $K,K'$, then there are a neighborhood $\mathcal{U}$ of $A$ and cones $\widetilde{K},\widetilde{K}'$ such that any $B\in\mathcal{U}$ satisfies the cone condition with cones $\widetilde{K},\widetilde{K}'$. In particular, applying Lemma \ref{lem:cone-condition} we find that there exists a constant $c=c(\mathcal{U})\in (0,1)$, such that if $B=(B_1,\ldots, B_m)\in\mathcal{U}$, then
\begin{equation} \label{eq:cone-supperadditivity}
\| B_\i B_\j \| \ge c\, \|B_\i\| \| B_\j\| \quad\text{for all }\i,\j\in I^*.
\end{equation}
Now, for this constant $c$, let
\[
\widetilde{S}_n(B,s) = c \sum_{\i\in I^n} \| B_\i\|^s,
\]
and observe that if $B\in\mathcal{U}$ then, thanks to \eqref{eq:cone-supperadditivity}, $\widetilde{S}_{n+k}(B,s)\ge \widetilde{S}_n(B,s)\widetilde{S}_k(B,s)$. Therefore, for $B\in\mathcal{U}$,
\[
M(B,s) = \lim_{n\to\infty} \frac1n \log\widetilde{S}_n(B,s) = \sup_n \frac1n \widetilde{S}_n(B,s).
\]
Since a supremum of continuous functions is lower semicontinuous, this yields the claim for $M$.

Suppose now $d=2$. Since $\alpha_2(B)=\|B^{-1}\|^{-1}$ for $B\in\text{GL}_2(\R)$, we have that $\alpha_2(B_1 B_2) \ge \alpha_2(B_1)\alpha_2(B_2)$ for any $B_1,B_2\in\text{GL}_2(\R)$. Let $\widetilde{K},\widetilde{K}',\mathcal{U},c$ be as before. Then
\[
\phi^s(B_\i B_\j) \ge c\, \phi^s(B_\i) \phi^s(B_\j) \quad\text{for all }\i,\j\in I^*.
\]
Thus, arguing as before,
\[
P(B,s) = \sup_n \frac1n \log\left( \sum_{\i\in I^n} c\, \phi^s(B_\i)\right),
\]
which is lower semicontinuous as a supremum of continuous functions.
\end{proof}
Although we will not use this result directly, the ideas in its proof will arise in our proof of continuity of $M$ and $P$ in dimension $d=2$.

\section{Proof of the continuity of subadditive pressure in $\R^2$}

\label{sec:proof}

\subsection{General strategy and the case of equal Lyapunov exponents}

\label{subsec:equal-exponents}

In this section we prove Theorems \ref{thm:matrix-pressure-continuity} and \ref{thm:continuity-pressure} in dimension $d=2$ (recall that Theorem \ref{thm:affinity-dim-continuous} is an immediate corollary of Theorem \ref{thm:continuity-pressure}). We are going to give the details of the continuity of $P(A,s)$; the proof for $M(A,s)$ is essentially identical. We have already observed that $P$ is upper semicontinuous, hence it is enough to prove it is lower continuous. Moreover, by the second part of Lemma \ref{lem:properties-SVF-pressure}, it is enough to prove continuity in $A$ for a fixed value of $s$.

Fix $\e>0$ for the course of the proof. Consider the potential $\Phi=\{\phi_n\}$ where $\phi_n(\i)=\phi^s(A(\i,n))$ (this is the potential given in \eqref{eq:potential-SVF}). Thanks to the variational principle for subadditive potentials (Theorem \ref{thm:variational-principle}), we know that there exists an ergodic measure $\mu$ on $X$, such that
\begin{equation} \label{eq:application-variational-principle}
h_\mu + E_\mu(\Phi) \ge P(\Phi)-\e = P(A,s)-\e.
\end{equation}
(In fact, K\"{a}enm\"{a}ki \cite{Kaenmaki04} showed we can take $\e=0$ in the above, but we do not need this). The potential $\Phi$ and the measure $\mu$ will remain fixed for the rest of the proof; we underline that they depend on $s$ and $A$.

We apply Oseledets' Theorem (Theorem \ref{thm:oseledets}) to the the matrices $(A_1,\ldots,A_m)$ and the measure $\mu$. The proof splits depending on whether the resulting Lyapunov exponents are equal or distinct. However, in both cases we will rely on the general scheme given in the next lemma.

\begin{lemma} \label{lem:general-scheme}
Suppose there are $n=n(\e)$, $Y_n\subset I^n$, and a neighborhood $\mathcal{U}$ of $A$ such that the following hold:
\begin{enumerate}
\item \label{it:Yn-large} $\log |Y_n| \ge n(h_\mu- \rho_1(\e))$,\\
\item \label{it:Yn-SVF-large} If $\i$ is a juxtaposition of $k$ words from $Y_n$, and $B\in\mathcal{U}$, then
\[
\log \phi^s(B_\i) \ge nk (E_\mu(\Phi)-\rho_2(\e)).
\]
Then $P(B,s) \ge P(A,s)-\e-\rho_1(\e)-\rho_2(\e)$ for all $B\in\mathcal{U}$.
\end{enumerate}
\end{lemma}
\begin{proof}
Let $Y_n^k$ denote the family of juxtapositions of $k$ words from $Y_n$. If $B\in\mathcal{U}$, then
\begin{align*}
P(B,s) &\ge \limsup_{k\to\infty} \frac{1}{nk}\log \sum_{\i \in Y_n^k} \phi^s(B_\i)\\
&\ge \lim_{k\to\infty} \frac{1}{nk} \left(k\log|Y_n|+\min_{\i\in Y_n^k}\log \phi^s(B_\i)\right)\\
&\ge h_\mu + E_\mu(\Phi) - \rho_1(\e)-\rho_2(\e) \\
&\ge P(A,s)-\e- \rho_1(\e)-\rho_2(\e),
\end{align*}
where in the last line we have used \eqref{eq:application-variational-principle}.
\end{proof}
In practice we will take $\rho_i(\e)$ to be a multiple of $\e$, so that in the limit as $\e\to 0$ we obtain the required lower semicontinuity. Finding a set $Y_n$ such that \eqref{it:Yn-large} holds is not difficult, and likewise if we also require \eqref{it:Yn-SVF-large} only for $\i\in Y_n$ (rather than $Y_n^k$). The challenge is to make \eqref{it:Yn-SVF-large} stable under compositions of the maps $B_\j, \j\in Y_n$ as well, and for this we will need geometric and ergodic-theoretic arguments depending on Oseledets' Theorem.

First we deal with the simpler case in which the Lyapunov exponents are equal; the case of different exponents is addressed in the next subsection.

Suppose then that there is a single Lyapunov exponent $\lambda$. It follows easily from \eqref{eq:subadditive-ergodic-theorem} and Theorem \ref{thm:oseledets} that
\begin{equation} \label{eq:energy-case-equal-Lyap-exponents}
\lim_{n\to\infty} \frac1n \int \log \phi^s( A(\i,n))d\mu(\i) = s\log\lambda.
\end{equation}

By Theorem \ref{thm:oseledets}, \eqref{eq:ShannonMcMillanBreiman} and Egorov's Theorem, we can find a set $Y\subset X$ such that $\mu(Y)\ge 1/2$, and $n_0\in\N$ such that if $\i\in Y$ and $n\ge n_0$ then
\begin{align}
| A(\i,n) v| &\ge \lambda^n e^{-\e n} |v| \quad\text{for all }v\in\R^2\setminus\{0\}, \label{eq:Egorov-Oseledets}\\
\mu[\i|n] &\le  e^{-n (h_\mu-\e)} . \nonumber
\end{align}
Fix some $n\ge n_0$ and write $Y_n = \{ i|n: \i\in Y\}$. Note that
\[
\frac12 \le \sum_{\j\in Y_n} \mu[\j] \le |Y_n| e^{-n (h_\mu-\e)},
\]
whence (if $n$ is large enough)
\[
|Y_n| \ge e^{n (h_\mu-2\e)}.
\]

On the other hand, since $Y_n$ is finite, we can find a neighborhood $\mathcal{U}$ of $A$ such that if $B=(B_1,\ldots, B_m)\in\mathcal{U}$, then
\[
| B_{\j} v| \ge \lambda^n e^{-2\e n} |v| \quad\text{for all }v\in\R^2\setminus\{0\},\j\in Y_n.
\]
Therefore
\[
| B_{\i} v| \ge \lambda^{kn}e^{-2\e kn} |v|\quad\text{for all }v\in\R^2\setminus\{0\},\i\in Y_n^k,
\]
where $Y_n^k\subset I^{kn}$ is the set of all juxtapositions of $k$ words from $Y_n$. In particular,
\[
\phi^s(B_\i) \ge \lambda^{kns}e^{-2\e kns}\quad\text{for all }\i\in Y_n^k.
\]
We have therefore established the hypotheses of Lemma \ref{lem:general-scheme}, with $\rho_1(\e)=2\e$ and $\rho_2(\e)=2s\e$. Applying that lemma and letting $\e\to 0$ establishes lower semicontinuity when the Lyapunov exponents are equal.

\subsection{The case of distinct Lyapunov exponents}

Suppose now that the Lyapunov exponents are $\lambda_1>\lambda_2$. We will again construct sets $Y_n$ so that we can apply Lemma \ref{lem:general-scheme}; this is trickier in this case, and the main idea is to use Oseledet's Theorem, Egorov's Theorem and recurrence, to find sets $Y_n$ (for $n$ arbitrarily large) so that the hypotheses of Lemma \ref{lem:general-scheme} hold when $k=1$, \emph{and in addition $\{ A_\i: \i\in Y_n\}$ satisfies the cone condition}. The cone condition will allow us to pass to a neighborhood of $A$ first, and to iterates of the $B_\i$, $\i\in Y_n$, later.

Recall that for $\mu$-almost all $\i$ there is an Oseledets splitting $\R^2= E_1(\i)\oplus E_2(\i)$. The family of splittings $\R^2=E_1\oplus E_2$ has a natural separable metrizable topology; for example, we can take $d(E_1\oplus E_2, E'_1\oplus E'_2)= \max(\angle(E_1,E'_1), \angle(E_2,E'_2))$, where $\angle$ is the angle between two lines. We can then find a fixed splitting $\R^2= F_1\oplus F_2$ which is in the support of the push-forward of $\mu$ under the Oseledets splitting or, in other words,
\[
\mu(\i: d(E_1(\i)\oplus E_2(\i), F_1\oplus F_2)<\eta) > 0 \quad\text{for all }\eta>0.
\]
We write $F_i^\gamma = \{ E: \angle(E,F_i)<\gamma\}$.

\begin{lemma} \label{lem:cone-condition-oseledets}
There are $R,\eta>0$ and two cones $K',K\subset\R^2$ with $K'\setminus\{0\}\subset K$, such that the following holds. Suppose that $A\in\text{GL}_2(\R)$ is such that $A v_j \in F_j^\eta$ for some $v_j\in F_j^\eta$ of unit norm, $j=1,2$, and moreover $|A v_1| > R |A v_2|$. Then $AK\subset (K'\cup K')$.
\end{lemma}
\begin{proof}
The lemma is essentially a consequence of compactness. Let $v$ be a unit vector in $F_1$, and let $K, K'$ be any cones such that
\[
v\in \text{interior}(K')\setminus\{0\} \subset K'\setminus\{0\} \subset K \subset \R^2\setminus F_2.
\]
Suppose the claim fails with this choice of cones. Then for each $n$ there are $A_n\in\text{GL}_d(2)$ and $v_{n,j}\in F_j^{1/n}$ such that
\begin{equation} \label{eq:cone-in-limit}
1 = | A_n v_{n,1} | \ge n | A_n v_{n,2} |,
\end{equation}
and $A_n K \not\subset K'\cup -K'$. By passing to a subsequence (and replacing $v_{n,1}$ by $-v_{n,1}$ whenever needed), we may assume that $v_{n,1}\to v$, $v_{n,2}\to w$ and $A_n\to A$ for some $w\in F_2$ of unit norm, and $A\in \R^{2\times 2}$. Moreover, there is $z\in K$ of unit norm such that $A z\notin \text{interior}(K'\cup K')$. However, \eqref{eq:cone-in-limit} implies that $Az$ is a non-zero multiple of $v$ for any $z\notin F_2$, which contradicts our choice of cones. This contradiction finishes the proof of the lemma.
\end{proof}

From now on let $R,\eta, K, K'$ be as in the statement of the Lemma. By our choice of $F_1, F_2$, we have $\mu(\Delta)>0$, where
\[
\Delta=\{ \i\in X: E_j(\i)\in F_j^\eta, j=1,2 \}.
\]
At this point we recall the following quantitative version of Poincar\'{e} recurrence due to Khintchine, see \cite[Theorem 3.3]{Petersen89} for a proof. Although it applies to any set of positive measure in a measure-preserving system, we state only the special case we will require.

\begin{lemma}
For every $\delta>0$, the set $\{ n: \mu(\sigma^{-n}\Delta\cap \Delta) > \mu(\Delta)^2-\delta \}$ is infinite (and it even has bounded gaps).
\end{lemma}
In particular, if we set $\kappa:=\mu(\Delta)^2/2>0$, then the set $S:=\{ n: \mu(\sigma^{-n}\Delta\cap \Delta) > \kappa \}$ is infinite. On the other hand, by \eqref{eq:ShannonMcMillanBreiman}, \eqref{eq:subadditive-ergodic-theorem}, the last part of Oseledets' Theorem, and Egorov's Theorem, we may find $n_0\in\mathbb{N}$ and a set $\Sigma\subset X$ with $\mu(\Sigma)>1-\kappa/2$, such that if $n\ge 0$ and $\i\in\Sigma$, then:
\begin{align}
\mu[\i|n] &\le e^{-n (h_\mu-\e)}, \label{eq:measure-upper-bound}\\
\log\phi^s(A(\i,n)) &\ge n(E_\mu(\Phi)-\e), \label{eq:energy-lower-bound}\\
|A(\i,n)\widehat{E}_1(\i)| &\ge R \, |A(\i,n)\widehat{E}_2(\i)|,\nonumber
\end{align}
where $\widehat{E}_j(\i)$ is a unit vector in $E_j(\i)$. This is the point where we use that the Lyapunov exponents are different.

Taking stock, we have seen that if $n\ge n_0$, and $\i \in \Delta\cap \sigma^{-n}\Delta\cap\Sigma$ then, by Lemma \ref{lem:cone-condition-oseledets}, the map $A(\i,n)$ satisfies
\[
A(\i,n)K\subset (K'\cup -K').
\]
Hence for $n\in S\cap [n_0,\infty)$, we define $Y_n=\{ \i|n: \i\in\Delta\cap \sigma^{-n}\Delta\cap\Sigma\}$. We will show that these sets meet the conditions of Lemma  \ref{lem:general-scheme}, with suitable functions  $\rho_j(\e)$. Firstly, note that
\begin{align*}
\kappa/2 &\le \mu(\Delta\cap \sigma^{-n}\Delta\cap \Sigma)\\
&\le \sum_{\j\in Y_n} \mu[\j] \\
&\le |Y_n| e^{-n (h_\mu-\e)}.
\end{align*}
Hence $\log|Y_n|\ge n(h_\mu-2\e)$, provided $n$ is taken large enough that $e^{-\e n}<\kappa/2$.

On the other hand, $\{A_\j: \j\in Y_n\}$ satisfies the cone condition with cones $K, K'$ (these cones are independent of $n$). Then there are a neighborhood $\mathcal{U}$ of $A$ in $(\text{GL}_d(\R))^m$ and cones $\widetilde{K},\widetilde{K'}$ such that if $B\in\mathcal{U}$, then $\{ B_{j_n}\cdots B_{j_1} : (j_1\ldots j_n)\in Y_n\}$ satisfies the cone condition with cones $\widetilde{K},\widetilde{K'}$. In particular, by Lemma \ref{lem:cone-condition}, there exists $c>0$ (depending only on $\mathcal{U}$, and not on $n$) such that
\[
\| B_\i \| \ge c^{k-1} \left(\min_{\j\in Y_n} \|B_\j\|\right)^k\quad\text{for all }\i\in Y_n^k.
\]
Arguing as in the proof of Lemma \ref{lem:continuity-under-cone-condition},
\[
\phi^s(B_\i) \ge c^{k-1} \left(\min_{\j\in Y_n} \phi^s(B_\j)\right)^k\quad\text{for all }\i\in Y_n^k.
\]
By taking $n$ large enough, we may assume that $\log c/n > -\e$. Furthermore, in light of \eqref{eq:energy-lower-bound} we may find a neighborhood $\mathcal{V}\subset\mathcal{U}$ containing $A$, such that if $B\in\mathcal{V}$ and $\j\in Y_n$, then $\log \phi^s(B_\j) > n(E_\mu(\Phi)-2\e)$. We conclude that if $B\in\mathcal{V}$ and $\i\in Y_n^k$, then
\[
\log \phi^s(B_\i) > kn(E_\mu(\Phi)-3\e).
\]
We are now able to apply Lemma \ref{lem:general-scheme} to conclude that if $B\in\mathcal{V}$, then
\[
P(B,s) \ge P(A,s) - 6\e.
\]
As $\e>0$ was arbitrary, this yields the desired lower semicontinuity.

\subsection{Some remarks on the higher-dimensional case}

We finish the paper with some brief remarks on the proof of the continuity of $M$ and $P$ in any dimension.

The proof of the continuity of $M$ in dimension $2$ extends fairly easily to arbitrary dimension $d$ (using the general version of Oseledets' Theorem): if all $d$ Lyapunov exponents are equal, then the proof is identical to the two-dimensional case. If not all exponents are equal, let $1\le k<d$ be the multiplicity of the largest Lyapunov exponent. Then the argument is very similar, except that one uses cones around $k$-planes.

For $d\ge 3$, one cannot reduce $\phi^s$ to a quantity involving only matrix norms (of the given maps and their inverses), so it is clear that some new tools are required to prove continuity of $P$ in general dimension. The proof follows the same outline, but it involves working with higher exterior powers of the maps $A_i$, and proving cone conditions for two different exterior powers simultaneously. Altough passing to exterior products is a common trick in the area, this makes the general proof far more technical. The reader is referred to \cite{FengShmerkin13} for further details, as well as consequences and generalizations of these results.


\end{document}